\documentclass[12pt]{amsart}
\usepackage{amsmath,amscd,amsthm,amsfonts, amssymb,amsxtra}
\usepackage{calrsfs}
\usepackage[all]{xy} \SelectTips{cm}{}
\usepackage{enumitem}
\usepackage{hyperref}
  \hypersetup{colorlinks=true,citecolor=blue}
  \usepackage{graphicx}
\CDat
\newtheorem{thm}{Theorem}[section]  
\newtheorem*{un-no-thm}{Theorem}
\newtheorem{cor}[thm]{Corollary}     
\newtheorem{lem}[thm]{Lemma}         
\newtheorem{prop}[thm]{Proposition}  

\newtheorem{conjecture}[thm]{Conjecture}
\newtheorem{ques}[thm]{Question}
\newtheorem{bigthm}{Theorem}
\newtheorem{bigcor}[bigthm]{Corollary}

\theoremstyle{definition}
\newtheorem{defn}[thm]{Definition}   

\theoremstyle{definition}

\theoremstyle{definition}
\theoremstyle{remark}
\newtheorem{rem}[thm]{Remark}
\newtheorem{rems}[thm]{Remarks}

\newtheorem{hypo}[thm]{Hypothesis}
\newtheorem{notation}[thm]{Notation}
\newtheorem*{acks}{Acknowledgements}

\newtheorem*{cl}{Claim} 
\newtheorem*{out}{Outline}

\newtheorem*{intro-rem}{Remark}
\newtheorem*{intro-rems}{Remarks}

\newtheorem{ex}[thm]{Example}

\DeclareMathOperator*{\holim}{holim}

\DeclareMathOperator*{\colim}{colim}

\DeclareMathOperator{\secs}{sec}
\DeclareMathOperator{\lift}{lifts}

\DeclareMathOperator{\pd}{pd}
\DeclareMathOperator{\sm}{sm}
\DeclareMathOperator{\bl}{bl}
\DeclareMathOperator{\st}{st}

\begin{document}
\title{Embeddings, normal invariants and
functor calculus}
\date{\today} 
\author{John R. Klein} 
\address{Department of Mathematics, Wayne State University,
Detroit, MI 48202} 
\email{klein@math.wayne.edu} 

\begin{abstract}  This paper investigates the space of codimension zero embeddings
of a Poincar\'e duality space in a disk. One of our main results
exhibits a tower that interpolates from
the space of Poincar\'e immersions to a certain space of ``unlinked''
Poincar\'e embeddings. The layers of this tower are described in terms
of the coefficient spectra of the identity appearing in Goodwillie's
homotopy functor calculus. We also answer a question posed to us by Sylvain Cappell.
The appendix proposes
a conjectural relationship between our tower and the manifold calculus tower for 
the smooth embedding space.
 \end{abstract}

\maketitle
\setlength{\parindent}{15pt}
\setlength{\parskip}{1pt plus 0pt minus 1pt}
\def\Top{\bold T\bold o \bold p}
\def\wTop{\text{\rm w}\bold T}
\def\wT{\text{\rm w}\bold T}
\def\vo{\varOmega}
\def\vs{\varSigma}
\def\smsh{\wedge}
\def\esmsh{\,\hat\wedge\,}
\def\flush{\flushpar}
\def\dbslash{\smsh!\! /}
\def\:{\colon\!}
\def\Bbb{\mathbb}
\def\bold{\mathbf}
\def\cal{\mathcal}
\def\orb{\cal O}
\def\hoP{\text{\rm ho}P}

\setcounter{tocdepth}{1}
\tableofcontents
\addcontentsline{file}{sec_unit}{entry}

\section{Introduction \label{sec:intro}}

\subsection{Background} Suppose
that $P$ and $N$ are compact smooth $n$-manifolds, possibly with boundary.
Let $E^{\sm}(P,N)$ be the space of smooth ($C^\infty$) embeddings
from $P$ into the interior of $N$.
The {\it manifold calculus} of Goodwillie and Weiss produces a tower  
of fibrations
\[
\cdots \to E^{\sm}_2(P,N) \to E^{\sm}_1(P,N)
\]
and compatible maps $E^{\sm}(P,N) \to E^{\sm}_j(P,N)$. 
If we assume that $P$ admits a handle decomposition 
with handles of index at most $n-3$, then the maps $E^{\sm}(P,N) \to E^{\sm}_j(P,N)$
have connectivity given by a linear function of $j$
with positive slope, so in this case the tower strongly converges \cite{GW}, \cite{GK2}.
Furthermore, $E^{\sm}_1(P,N)$ 
has the homotopy type of the space of smooth immersions from $P$ to $N$. 
For $j \ge 2$, the layers
of the tower, i.e., the homotopy fibers of the maps $E^{\sm}_j(P,N) \to E^{\sm}_{j-1}(P,N)$,
have an explicit description in terms of configuration spaces.

In essence, the strong convergence result relies on the following schematic passage:
\[
E^{\sm}(P,N) \to E^{\bl}(P,N) \to E^{\pd}(P,N)\, ,
\]
where $E^{\bl}(P,N)$ is the space of smooth block embeddings of $P$ in $N$ and
$E^{\pd}(P,N)$ is the corresponding space of Poincar\'e embeddings. 
Convergence is proved by establishing certain higher excision statements,
which are known as ``multiple disjunction'' results for spaces of smooth embeddings. 
One achieves such results by first
 proving analogous ones for spaces of Poincar\'e embeddings. The Poincar\'e
 statements were
proved in \cite{GK1} using homotopy theory. One then lifts the Poincar\'e statements
to the block setting using surgery theory. The final step is to lift the block statements
to the smooth ones using concordance theory. Given the method of proof, it
seems  appropriate to ask:

\begin{ques} \label{ques:q1} Is there an analogue of Goodwillie-Weiss 
manifold calculus in the Poincar\'e duality space setting?
\end{ques}

More precisely, suppose now that $P$ and $N$ are Poincar\'e spaces 
of dimension $n$ (possibly with boundary).
One then has a space of Poincar\'e embeddings $E^{\pd}(P,N)$ and we wish to 
construct a  Goodwillie-Weiss calculus for it. Unfortunately, we do not
know how to proceed. The problem
here is that the set-up of \cite{Weiss} does not properly translate over: in the manifold case one considers the poset 
of subsets of the interior of $P$ which are diffeomorphic to finite collections of open balls. This poset has good properties because a manifold is locally Euclidean.
In the Poincar\'e case there does not seem to be a sensible replacement for this,
as Poincar\'e spaces are not necessarily locally well-behaved. 

A related but perhaps more accessible question is

\begin{ques}  \label{ques:q2}
Is there a version of the Goodwillie-Weiss tower in the Poincar\'e embedding case?
\end{ques} 

We propose to attack Question \eqref{ques:q2}
from a point-of-view arising out of the surgery school in conjunction with one of the
other functor calculi: Goodwillie's {\it homotopy functor calculus.}

To simplify the presentation, we will only consider the case when $N=D^n$ is an 
$n$-disk, and we will assume that $P$ is ``sectioned'' in the sense defined below.
We  will see that a certain space 
\[
\frak LE^{\pd}(P,D^n) 
\]
of ``unlinked'' Poincar\'e embeddings of $P$ in $D^n$ does have 
a tower associated with it. A point in this space consists of a Poincar\'e embedding of
$P$ in $D^n$ together with a choice of nullhomotopy of the ``link'' of the embedding.
We will also see that the tower associated with this space
strongly converges under mild hypotheses,
and its first stage coincides with the space of 
``Poincar\'e immersions'' of $P$  in $D^n$. Furthermore, we will identify the 
homotopy fibers of the canoncial map $\frak LE^{\pd}(P,D^n) \to E^{\pd}(P,D^n)$ as spaces of ``unlinkings'' of $P$ in $D^n$.

In what follows we simplify
notation by setting
\[
E(P,D^n) := E^{\pd}(P,D^n)\, .
\]

\subsection{Sectioned Poincar\'e spaces}
Our goal will be to say something
sensible about the embeddings of the following class of Poincar\'e spaces:

\begin{defn} A {\it sectioning} of a Poincar\'e space $P$ with boundary $\partial P$
is
a triple
\[
\xi = (K,f,s)
\]
in which
\begin{itemize}
\item $K$ is a cofibrant space;
\item $f\: K @> \simeq >> P $ is a homotopy equivalence;
\item $s\: K \to \partial P$ is a map such that the composition
\[
K @> s >> \partial P \to P
\]
coincides with $f$. 
\end{itemize}
We refer to $\xi$ as {\it sectioning data.} For the sake of brevity, we will
say that $P$ is {\it sectioned} when the sectioning data
are understood.
\end{defn}

\begin{ex} \label{ex:timesI} Let $D^1 = [-1,1]$ be the 1-disk. 
If $Q$ is a Poincar\'e space, possibly with boundary
$\partial Q$, then 
$Q\times D^1$ is  sectioned by means of the homotopy equivalence
$Q\times \{-1\} \subset Q\times D^1$ and the inclusion
$Q\times \{-1\}\subset \partial (Q\times D^1)$.
\end{ex}

\begin{ex} \label{ex:spherical_fibration}
Suppose $\eta$ is a $(j-1)$-spherical fibration over 
a  Poincar\'e space $Q$ of dimension $d$ having empty boundary. 
Let $S(\eta)$ be its total space. 
Suppose $\eta$ comes equipped with a section $s\: Q \to S(\eta)$. 
Let $D(\eta)$ be the mapping cylinder of $\eta$, and let
$f\: Q \to D(\eta)$ be the inclusion. Then $\xi := (Q,f,s)$ is 
sectioning data for the $(n+j)$-dimensional Poincar\'e space $D(\xi)$.
\end{ex}

\begin{defn}[Generalized Thom Space]
\label{defn:thom} If $P$ is sectioned by $\xi = (K,f,s)$, 
then we define
\[
P^\xi := \partial P \cup_s CP\, ,
\]
i.e., the mapping cone of the map $s\: K\to \partial P$. This is a based space.
\end{defn}

The justification  for this notation/terminology is that 
the spherical fibration case in Example \ref{ex:spherical_fibration} gives
the Thom space in the usual sense. 

\begin{lem} Assume that $P$ is sectioned by $\xi = (K,f,s)$ and
$\partial P \to P$  is a cofibration. Then there is a preferred weak homotopy equivalence
\[
\Sigma P^\xi \simeq P/\partial P\, .
\]
That is, $P^\xi$ is a preferred desuspension of $P/\partial P$.
 \end{lem}
 
 \begin{proof}  One has a commutative diagram
 \[
 \xymatrix{
 K \ar[r]^f_{\sim} \ar[d]_s & P \ar[r]\ar@{=}[d] & P \cup_K CK\ar[d] \\
 \partial P \ar[r]\ar[d] & P \ar[r] \ar[d] & P\cup_{\partial P} C\partial P  \ar[d]^{\sim} \\
P^\xi \ar[r] & CP \ar[r] & CP \cup_{P^\xi} CP^\xi
\\}
\]
in which $CX$ denotes the cone on a space $X$. The 
rows and columns of the diagram form homotopy cofiber sequences (cf.~2.1; the null homotopies in this case are evident). The space in the upper right corner is contractible
so the map $P\cup_{\partial P} C\partial P \to CP \cup_{P^\xi} CP^\xi$ is a weak equivalence. The domain of this  map is identified with $P/\partial P$ up to a
preferred weak equivalence (given by collapsing $CP$ to a point) and the codomain is identified with $\Sigma P^{\xi}$ (an explicit identification is obtained using
the evident map $CP \to CP^\xi$ to get a weak equivalence $CP \cup_{P^\xi} CP^\xi \to 
 CP^\xi \cup_{P^\xi} CP^\xi = \Sigma P^\xi$). 
\end{proof}

In particular, the lemma gives a preferred isomorphism of singular
homology groups
$\tilde H_{k-1}(P^\xi) \cong H_k(P,\partial P)$.

\subsection{Homotopy codimension}
A  Poincar\'e space $P$ of dimension $n$ is said to have {\it homotopy codimension} 
$\ge j$ if the map $\partial P \to P$ is $(j-1)$-connected.

\begin{ex} Let $\eta\: S(\eta) \to Q$ be
a $(j-1)$-spherical fibration over a Poincar\'e space
without boundary. Then $D(\eta)$ has homotopy codimension $\ge j$.
\end{ex}

\begin{ex} If $P$ is a compact smooth manifold, possibly with boundary,
which admits a handle decomposition
whose handles all have index $\le k$, then the homotopy codimension of $P$
is $\ge n-k$.
\end{ex}

\begin{ex} \label{codim3}
Suppose $P$ is an $n$-dimensional Poincar\'e space, $n-k \ge 3$. Then
$P$ has homotopy codimension $\ge n-k$ if the map
$\partial P \to P$ is $2$-connected and $P$ has the weak homotopy
type of a CW complex of dimension $\le k$. This is a consequence of duality,  
the relative Hurewicz theorem and a result of Wall \cite[thm.~E]{Wall_finiteness}.
 \end{ex}

We will assume the following throughout the paper.

\begin{hypo}  The Poincar\'e space $P$  has homotopy codimension $\ge 3$.
\end{hypo}

\subsection{Poincar\'e embeddings} We recall
the notion of codimension zero Poincar\'e embedding
(see e.g., \cite{Klein_compress},\cite{GK1}). We will
restrict ourselves to the case when the ambient space is an $n$-disk.

A {\it Poincar\'e embedding} of a Poincar\'e space $P$
of dimension $n$ in $D^n$ consists of  a ``complement''
space $C$ equipped with a (gluing data) 
map 
\[
\partial P \amalg S^{n-1} \to C
\]
making $C$ into an $n$-dimensional Poincar\'e space with 
boundary $\partial P \amalg S^{n-1}$.
Furthermore, we require the homotopy pushout of
\[
P @<<< \partial P @>>> C
\]
to have weak homotopy type of $D^n$, i.e., it is required to be weakly contractible. 
The set of all such Poincar\'e embeddings comes equipped with
a topology (cf.\ \S\ref{sec:embeddings}). 

We denote this space by $E(P,D^n)$. 
We typically specify a Poincar\'e embedding  by writing its complement, i.e., we write
$C \in E(P,D^n)$.

\subsection{Unstable normal invariants} 
Assume $P$ is sectioned by $\xi$. 

\begin{defn} An {\it unstable normal invariant} for $P$ is a based map 
\[
\alpha\: S^{n-1} \to P^\xi 
\] 
such that 
\[
\alpha_\ast([S^{n-1}]) \in \tilde H_{n-1}(P^\xi) \cong H_n(P,\partial P)
\]
is a fundamental class for $P$.
\end{defn}

\begin{rem} The stable version of the normal invariant appeared in the 
context of surgery theory \cite{Novikov}. Applications of unstable normal invariants to embedding theory
were investigated in \cite{Williams79},
\cite{Richter} and \cite{Klein_compress}.
\end{rem}

\begin{prop} \label{prop:the-browder-construction} Assume $P$ is sectioned.
Then an unstable normal invariant for $P$ gives rise to a Poincar\'e embedding
of $P$ in $D^n$.
\end{prop}

\begin{proof} The proof harkens back 
to a construction of Browder \cite{Browder66}. Let $\xi$ be the sectioning data.
Set $C := P^\xi$.  
We have an evident map
\[
P^\xi \vee S^{n-1} \to C
\]
where on the second wedge summand we use $\alpha$. Let $\partial P \to P^\xi$ be the evident map. Consider the composite
\[
\partial P \amalg S^{n-1} \to P^\xi \vee S^{n-1} \to C\, .
\]
It follows from \cite[lem.~2.3]{Klein_haef} that this map gives $C$ the structure of a Poincar\'e space with boundary $\partial P \amalg S^{n-1}$ and defines a Poincar\'e embedding of $P$ in 
$D^n$. (Note: since $(P,\partial P)$ is $2$-connected
and $P \cup_{\partial P} C$ is $\infty$-connected, it follows that $C$ is
$1$-connected. So Poincar\'e duality for $C$ only needs to be verified with constant
$\Bbb Z$ coefficients.)
\end{proof}

\begin{notation} If $P$ is sectioned by $\xi$, then the space of its unstable normal invariants
will be denoted by
\[
\Omega^{n-1}_{\perp}P^\xi\, .
\]
This is to be topologized
as a subspace
of the $(n{-}1)$-fold based loop space $\Omega^{n-1} P^\xi$.
\end{notation}

\begin{rem} The subspace  $\Omega^{n-1}_{\perp}P^\xi\subset
\Omega^{n-1}P^\xi$ is a collection of connected components:
the Hurewicz map gives a (homotopy) cartesian square
\[
\xymatrix{
\Omega^{n-1}_{\perp}P^\xi \ar[r]^{\subset}
 \ar[d] & \Omega^{n-1}P^\xi\ar[d] \\
 H_{n}(P,\partial P)^{\times} \ar[r]_{\subset} & H_{n}(P,\partial P)
 }
 \]
 where $H_{n}(P,\partial P)^{\times}$ is the set of fundamental classes
 of $P$. This set is nonempty if and only if $P$
 is orientable. If $P$ is connected and orientable then $H_{n}(P,\partial P)^{\times}$ 
has
 precisely two elements.
\end{rem}

The proof of Proposition \ref{prop:the-browder-construction}
yields a map
\begin{equation} \label{eqn:the-browder-construction}
\Omega^{n-1}_{\perp}P^\xi \to E(P,D^n)
\end{equation}
which we henceforth call the {\it Browder construction}.

\subsection{The link; the first main result}
We continue to assume that $P$ is sectioned by $\xi = (K,f,s)$.
Given a point $C\in E(P,D^n)$, we have a weak map
\begin{equation}\label{eqn:weak-map}
P @<f<{}^{\sim} < K @> s >> \partial P \to C\, .
\end{equation}
Using the basepoint of $S^{n-1}$ we obtain a preferred basepoint for $C$.
Let
\begin{equation}\label{eqn:link-invariant}
\ell_0(C) \in [P_+,C]_\ast \cong [P,C]
\end{equation}
be the homotopy class of the weak map \eqref{eqn:weak-map}, where $P_+$ is $P$ with a disjoint basepoint. We call $\ell_0(C)$ the {\it link} of the Poincar\'e embedding. The next
example motivates the
terminology.

\begin{ex} Let $f\: M^p \to D^{2p+1}$ be a smooth framed embedding, where $M$ is connected.
This means that $f$ admits a preferred extension 
to a smooth embedding $F\: M^p \times D^{p+1} \to D^{2p+1}$.
By Alexander duality and the Hurewicz isomorphism we have a preferred isomorphism
$H_p(D^{2p+1}\setminus f(M^p)) \cong \Bbb Z$. Let $\ast \in S^p$ be the basepoint. Then
the homology class induced by
\[
M^p \times \ast @>F >> D^{2p+1}\setminus f(M^p)
\]
is the {\it self-linking number} of $f$.
\end{ex}

\begin{bigthm}\label{bigthm:browder-fibration} Assume $P$ is sectioned by $\xi$.
Given a Poincar\'e embedding $C\in E(P,D^n)$, then
the homotopy fiber of the Browder construction \eqref{eqn:the-browder-construction} taken at $C$ is non-empty if and only if the link
$\ell_0(C)$ is trivial.

Furthermore, if the link is trivial, then the Browder
construction sits in a homotopy fiber sequence
\[
F_\ast(\Sigma (P_+),C) \to \Omega^{n-1}_{\perp}P^\xi \to E(P,D^n)
\]
where the homotopy fiber is taken at $C$ and
$F_\ast(\Sigma (P_+),C)$ is the function space of based maps 
$\Sigma (P_+) \to C$.
\end{bigthm}

\begin{rems} \label{rem:tp}
(1). The null homotopy yielding the homotopy fiber sequence 
will made explicit in the proof of Theorem \ref{bigthm:browder-fibration}.
\smallskip
 
(2). Theorem \ref{bigthm:browder-fibration} answers
a question posed to me by Sylvain Cappell about how far the Browder
construction is from being a homotopy equivalence.
\smallskip

\noindent (3). Let $Q$ be a 
Poincar\'e space of dimension $n-1$ and homotopy codimension $\ge n-k \ge 4$, but
not necessarily sectioned. Then as in Example \eqref{ex:timesI},
$Q{\times} D^1$ is sectioned with generalized Thom space $Q/\partial Q$.

The ``decompression'' map
$E(Q,D^{n-1}) \to E(Q{\times} D^1,D^{n})$ 
(cf.~\eqref{eqn:decompression} below) factors as
\[
E(Q,D^{n-1}) @>>> \Omega^{n-1}_\perp Q/\partial Q @>>> E(Q{\times} D^1,D^n)\, ,
\]
in which the first map is given by the Pontryagin-Thom construction,
and the second is the one of Theorem \ref{bigthm:browder-fibration}.
The results of \cite{Klein_compress}  imply that the first map 
is $(2n-3k-6)$-connected.
\smallskip

\noindent (4). The map
$
F_\ast(\Sigma (P_+),C) \to \Omega^{n-1}_{\perp}P^\xi 
$
can be described as an orbit map of an ``action'' of
$\Omega F_\ast(P_+,P^\xi)$ on $\Omega^{n-1}_{\perp}P^\xi$: 
 fix an unstable normal invariant $\alpha \: S^{n-1} \to P^\xi$. Then the Browder construction applied to $\alpha$ gives a Poincar\'e embedding with complement $C = P^\xi$.  

Let $
c\: P^\xi \to P^\xi \vee \Sigma (P_+)
$
be the Barratt-Puppe coaction map
for the based cofiber sequence $P_+ \to (\partial P)_+ \to P^\xi$.
Given  $\phi \in F_\ast(\Sigma (P_+),P^\xi) = \Omega F_\ast(P_+,P^\xi) $, we obtain a new
normal invariant $\phi\star \alpha$ by taking the composition
\[
S^{n-1} @> \alpha >> P^\xi @> c >> P^\xi \vee \Sigma (P_+)
@> \text{id} + \phi >> P^\xi \, .
\]
Then the operation $\phi \mapsto \phi\star \alpha$ yields the desired description.
\end{rems}

\begin{ex}[Embeddings of the $n$-disk] Let $P = D^n$ where $n \ge 3$.
Then $P$ is sectioned by the basepoint of $S^{n-1}$.
In this case, Theorem \ref{bigthm:browder-fibration} gives
a homotopy fiber sequence
\[
\Omega S^{n-1} \to F_{n-1} \to E(D^n,D^n)\, ,
\]
where $F_{n-1}$ is the space of based self homotopy equivalences of $S^{n-1}$.
In fact,  $E(D^n,D^n) \simeq G_n$, the
unbased self homotopy equivalences of $S^n$.
Furthermore, the above homotopy fiber sequence is principal and with respect to the identifications is just a shift to the left of the evident fibration
$F_{n-1} \to G_n \to S^{n-1}$.
\end{ex}

\begin{ex}[Disjoint unions of $n$-disks]\label{ex:little-disks} Assume $n \ge 3$. 
Let $T$ be a finite set and let $P = D^n \times T$. Then $P$ is sectioned
by choosing a basepoint $\ast \in S^{n-1}$.
Theorem \ref{bigthm:browder-fibration} gives in this case a homotopy fiber
sequence
\begin{equation}\label{eqn:little-disks}
\Omega \prod_T  S^{n-1} \smsh T_+ \to 
\Omega^{n-1}_\perp (S^{n-1}\smsh T_+) \to  E(D^n {\times} T,D^n)\, . 
\end{equation}
For each $x \in T$ we have a projection map $p_x\: S^{n-1}\smsh T_+ \to S^{n-1}$.
The condition for a based map $f\: S^{n-1} \to S^{n-1} \smsh T_+$
to be a normal invariant is that every composite $p_x \circ f$ should
lie in $F_{n-1}$. 

Since $S^{n-1}\smsh T_+$ is a finite wedge of spheres, 
the Hilton-Milnor theorem implies that the homotopy groups of
the two spaces on the left of \eqref{eqn:little-disks} can be expressed 
explicitly in terms
of the homotopy groups of spheres. 

For example, if $n$ is even, then  these two spaces are rationally equivalent
to generalized Eilenberg-Mac~Lane spaces with finitely generated homotopy groups that
only occur even degrees. These rational homotopy
groups can be explicitly computed in terms of
a Hall basis for the free Lie algebra on $|T|$-generators (see e.g., \cite[thm.~4.7]{Boardman-Steer}). For parity reasons,
the long exact homotopy sequence of rational homotopy groups splits
into short exact sequences

\[
\resizebox{5in}{!}{
\xymatrix{
0 \ar[r]  & \pi_{2k+1}(E(D^n \times T,D^n))_{\Bbb Q} \ar[r] &
\oplus_T  \pi_{2k+1} (S^{n-1} \smsh T_+)_{\Bbb Q} 
 \ar `r `[l] `[dll] [dll] 
\\
\pi_{2k+n-1}(S^{n-1}\smsh T_+)_{\Bbb Q} \ar[r]  &
\pi_{2k}(E(D^n \times T,D^n))_{\Bbb Q} \ar[r] & 0\, ,
}
}
\]
where we are assuming $k > 0$ (this is not
a serious restriction: it can be shown that $\pi_0(E(D^n \times T,D^n))$
is a set of cardinality $2^{|T|}$ and $\pi_1(E(D^n \times T,D^n))$ is
the direct sum of $|T|$-copies of the cyclic group of order two).
From this we immediately obtain a crude bound for the rank of the rational homotopy
groups of $E(D^n \times T,D^n)$. 
To obtain finer information would
require explicit knowledge of the curved arrow in the diagram.
\end{ex}

\subsection{Unlinked embeddings; the second main result}
\begin{defn} Assume $P$ is sectioned.
The space of {\it unlinked embeddings} 
\[
\frak L E(P,D^n)
\]
consists of those points 
$C \in E(P,D^n)$ such that the gluing data
\[
\partial P \amalg S^{n-1} \to C
\]
comes equipped with a factorization
\[ \partial P \amalg S^{n-1} \to P^\xi \vee S^{n-1} \to 
C
\]
where the first map is evident.
\end{defn}

It is clear that the Browder construction \eqref{eqn:the-browder-construction}
factors as
\begin{equation}\label{eqn:factorized-browder-construction}
\Omega^{n-1}_{\perp}P^\xi \to  \frak L E(P,D^n) \to E(P,D^n) \, .
\end{equation}
We call the map 
\begin{equation} \label{eqn:refined-browder-construction}
\Omega^{n-1}_{\perp}P^\xi \to  \frak L E(P,D^n)
\end{equation}
the {\it refined Browder construction.} 

\begin{bigthm} \label{bigthm:refined_browder-construction} 
The refined Browder construction \eqref{eqn:refined-browder-construction}
is a homotopy equivalence.
\end{bigthm}

\begin{rems} (1). Theorem \ref{bigthm:refined_browder-construction}  
is essentially a ``space-ification'' of a result Williams \cite[thm.~A(i)]{Williams79}
who considered a version of the map
\eqref{eqn:refined-browder-construction} on the level of path components. 
However, there are some minor differences:
\begin{itemize}
\item Williams restricts himself to sectioned Poincar\'e spaces of the type
appearing in
Example \ref{ex:spherical_fibration}.
\item Williams studies
Poincar\'e embeddings in the $n$-sphere rather than in the $n$-disk.
This distinction does not appear on the level of path components.
\item Williams equips his Poincar\'e spaces and their embeddings
with orientations. His unstable normal invariants are of degree one.
\end{itemize}
The disadvantage with the $n$-sphere is that Williams' result 
does not extend to higher homotopy groups without modifying the domain
of the refined Browder construction: the correct replacement is
the space of ``fiberwise unstable
normal invariants over $S^n$'' (cf.\ \cite{Klein_compress}).
Another disadvantage is that Williams 
has to work much harder than we do to prove his result.
\smallskip

\noindent (2). Let $Q$ be as in Remark \ref{rem:tp}(1).
Then Theorem \ref{bigthm:refined_browder-construction}
implies that the decompression map 
\[
E(Q,D^n) \to \frak LE(Q{\times} D^1,D^{n+1})
\]
is $(2n-3k-4)$-connected.
\end{rems}

\subsection{Poincar\'e Immersions} 
The space of (Poincar\'e) {\it immersions}
$I(P,D^n)$ is defined to be the homotopy colimit of the diagram
\[
E(P,D^n) \to E(P_1,D^{n+1}) \to E(P_2,D^{n+2})\to \cdots
\]
where $P_j = P \times D^j$. Note by construction 
\[
I(P,D^n) \simeq I(P_1,D^{n+1}) \simeq \cdots
\]
We view this as a reasonable definition, since
the analogous statement is valid in the smooth case.

We  will exhibit below a homotopy equivalence
\[
I(P,D^n) \simeq 
\Omega^{n}_{\perp} Q(P/\partial P)\, ,
\]
where the right-hand side denotes the space of {\it stable normal invariants} of $P$: this is the space whose points are stable maps $\alpha \: S^n \to P/\partial P$ such that
$\alpha_\ast([S^n])\in H_n(P,\partial P)$ is a fundamental class. We topologize
this as a subspace of $\Omega^{n}Q(P/\partial P) := \Omega^n\Omega^{\infty} \Sigma^\infty (P/\partial P)$.

\begin{rem}
Assuming  $I(P,D^n)$ is non-empty, then 
we will also exhibit (Lemma \ref{lem:smale-hirsch}) a homotopy equivalence
\[
I(P,D^n) \simeq F(P,G) \, ,
\]
where $G$ is the topological monoid of stable self-equivalences of the sphere.
The equivalence depends on choosing a basepoint in $I(P,D^n)$.
\end{rem}

\begin{rem} Assume $P$ is orientable.
The {\it Spivak normal fibration} of $P$ is an orientable stable spherical
fibration $\xi$ over $P$ that is equipped with a stable map 
$\alpha\: S^n \to P^\xi/(\partial P)^\xi$ representing a fundamental class
for $P$ by means of the Thom isomorphism.
The data $(\xi, \alpha)$ are well-defined up to contractible choice 
\cite{Spivak}, \cite{Klein_dualizing}.

Note that if $\xi$ is fiber homotopically trivial, then $\alpha$ amounts
to a stable normal invariant for $P$.
Furthermore, $I(P,D^n)$ is non-empty if and only the Spivak normal fibration
of $P$ is trivial.  The fiber homotopy triviality of $\xi$
is the Poincar\'e analog of stable parallelizability.
\end{rem}

\subsection{A tower for unlinked embeddings; the third main result}  
We introduce some notation. If $V$
is an orthogonal representation of a group $G$, we let $S^V$ be the
based $G$-sphere given by the one-point compactification of $V$.
If $V$ and $W$ are two orthogonal representations, then we write $V+W$ for the direct sum and $nV$ will denote the direct sum of $n$-copies of $V$.
If $W \subset V$ is an orthogonal sub-representation, then we let $V -W$
be its orthogonal complement. Let $1$ denote the trivial representation of rank one.
Let $\Sigma_j$ be the symmetric group on 
the standard basis for $\Bbb R^j$.  Then we obtain the {\it standard representation} 
of $\Sigma_j$ on $\Bbb R^j$. The diagonal gives an embedding of
the trivial representation $1$ inside the standard representation.
Let $V_j$ be denote its orthogonal complement. Call this the {\it reduced
standard representation;} it has rank $j-1$. For example, $V_2$ is the rank one
sign representation.

If $X$ is a based (cofibrant) space and $E$ is a (fibrant) spectrum, then
 we let $F_\ast(X,E)$ be the (stable) function spectrum whose $j$-th space is given by
 the based maps $X\to E_j$. If $X$ and $E$ are equipped with $G$-actions
 then $F_\ast(X,E)$ inherits a $G$-action by conjugation. In particular,
one can consider the homotopy orbit spectrum $F_\ast(X,E)_{hG}$ and its
associated infinite loop space $\Omega^\infty F_\ast(X,E)_{hG}$,
the latter which will be denoted by $F^{\st}_\ast(X,E)_{hG}$.

 Let $\cal W_j$ denote
 the $j$-coefficient spectrum of the identity functor
from based spaces to based spaces in the sense of the calculus of homotopy functors
\cite{Johnson-deriv}. In particular, $\cal W_j$ is a spectrum with $\Sigma_j$-action
which is unequivariantly weak equivalent to a wedge of $(j-1)!$ copies
of the $(1-j)$-sphere.

\begin{bigthm} \label{bigthm:EIpartial} 
Assume $P$ is sectioned and assume that the homotopy codimension of $P$ is $\ge n-k \ge 3$. 
Then there is a tower of fibrations
\[
\dots \to \frak L E_{j}(P,D^n) \to \frak L E_{j-1}(P,D^n) \to \cdots \to 
\frak L E_{1}(P,D^n)  
\]
equipped with compatible maps 
\[
\phi_j\: \frak L E(P,D^n) \to \frak L E_{j}(P,D^n)
\]
such that
\begin{itemize}
\item the map $\phi_j$ is $(2{-}n + (j{+}1)(n{-}k{-}2))$-connected; in particular, the 
induced map
\[
\frak L E(P,D^n) \to \lim_{j\to \infty} \frak L E_{j}(P,D^n)
\]
is a weak equivalence;
\item there is a preferred homotopy equivalence
 \[
 \frak L E_{1}(P,D^n) \simeq I(P,D^n)\, ;
 \]
\item if $j \ge 2$ and  $x\in \frak L E_{j-1}(P,D^n)$ is a point, then there
is an obstruction 
\[
\ell_{j-1}(x) \in  \pi_0(F_\ast(P^{\times j}_+,\cal W_j\smsh S^{(n-1)V_j+1})_{h\Sigma_j}) 
\]
which is trivial if and only of
the homotopy fiber of the map $\frak L E_{j}(P,D^n)\to \frak L E_{j-1}(P,D^n)$
at $x$ is non-empty;
 
\item if   $\ell_{j-1}(x)$ is trivial,
then there is a homotopy fiber sequence
\[
F^{\st}_\ast((P^{\times j})_+,\cal W_j \smsh S^{(n-1)V_j})_{h\Sigma_j} 
\to \frak L E_{j}(P,D^n)\to \frak L E_{j-1}(P,D^n)\, .
\]
where the displayed homotopy fiber is taken at $x$.
\end{itemize}
\end{bigthm}

\begin{rems} (1). The first part of  Theorem \ref{bigthm:EIpartial} 
implies that if $(j+1)k+2j \le jn$ and $\frak L E_j(P,D^n)$ is non-empty, then
$\frak L E(P,D^n)$ is also non-empty. 
\smallskip

\noindent (2).
It follows from
the last two parts of the theorem that the map
$\frak L E_{j}(P,D^n)\to \frak L E_{j-1}(P,D^n)$
is $(2{-}n + j(n{-}k{-}2))$-connected. 
\smallskip

\noindent (3).
Modulo torsion, a transfer argument shows that the class
 $\ell_{j-1}(x)$ is detected in the singular cohomology group
$H^{s}(P^{\times j};\Bbb Q^{(j-1)!})$, where $s = (n-2)(j-1)+1$. 
\smallskip

\noindent (4). The layers of the tower depend only on the homotopy type
of $P$ and in particular do not depend on the choice of sectioning data.
\end{rems}

\begin{out} The material of \S\ref{sec:prelim} is mostly language. The literate reader
can skip it and refer back to it as needed. In \S\ref{sec:embeddings} we provide
constructions of the Poincar\'e embedding space, its unlinked variant
and the space of Poincar\'e immersions. 
We prove Theorems \ref{bigthm:browder-fibration} and \ref{bigthm:refined_browder-construction} in \S\ref{sec:proofs}. In \S\ref{sec:proof-EIpartial} we prove
Theorem \ref{bigthm:EIpartial}. The final section, \S\ref{sec:relationship},
is conjectural: it poses a connection between the tower
of Theorem \ref{bigthm:EIpartial} and the Goodwillie-Weiss
tower for smooth embeddings.
\end{out}

\begin{acks} I learned about unstable normal invariants many years ago
in discussions with Bill Richter and Bruce Williams. 
Most of the research for this paper
was done while I visited the Mathematics Institute at
the University of Copenhagen in 
the Summer of 2014. 
\end{acks}

\section{Preliminaries \label{sec:prelim}}

\subsection{Spaces}
 Our ground category is $T$, the category of compactly generated weak Hausdorff spaces.
 A non-empty space $X$ is  {\it $r$-connected} if $\pi_j(X,x)$ is trivial
for $j \le r$, for all base points $x\in X$. The empty space is $(-2)$-connected
and every non-empty space is $(-1)$-connected. A map $X\to Y$ of non-empty spaces
is $r$-connected if each of its homotopy fibers is $(r-1)$-connected (every
map of non-empty spaces is $(-1)$-connected;
a weak homotopy equivalence is an $\infty$-connected map.

For unbased spaces $X$ and $Y$
we let $F(X,Y)$ for the unbased function space and if $X$ and $Y$ are based
we let $F_\ast(X,Y)$ be the 
 based function space. When we write $[X,Y]$, we mean homotopy classes of based maps
 $X^\text{c}\to Y$, where $X^\text{c}$ is a cofibrant replacement for $X$.
 When $X$ and $Y$ are based, then the based homotopy classes are to be written
 as $[X,Y]_\ast$. We use the usual notation for the smash product: $X \smsh Y$, and
 the iterated smash product of $j$-copies of $X$ is denoted $X^{[j]}$.

We equip $T$ with the  Quillen model category structure given by the Serre fibrations, Serre cofibrations and
 weak homotopy equivalences \cite{Quillen}, \cite[th.~2.4.23]{Hovey}. 
Note that $T$ is enriched over itself. We let $T_\ast$ denote the model category of based spaces. 

 A commutative square of spaces
 \[
 \xymatrix{
 X_{\emptyset} \ar[r] \ar[d] & X_2 \ar[d]\\
 X_1 \ar[r] & X_{12}
 }
 \]
 is {\it homotopy cocartesian} if the map 
 \[
 \text{hocolim}(X_1 \leftarrow X_{\emptyset}\to  X_{2}) \to X_{12}) \to X_{12}
 \]
 is  a weak equivalence, where the domain of this map
 is given by the homotopy pushout of the diagram obtained from the square by removing its terminal vertex. 
 In the special case when $X_2$ is contractible, we 
 abuse notation and refer to $X_\emptyset \to X_1 \to X_{12}$ as a {\it homotopy
 cofiber sequence.} The is the same
as equipping the composition
  $X_\emptyset \to X_1 \to X_{12}$  with a preferred choice of  null homotopy such that 
  the induced map $X_2 \cup_{X_\emptyset} CX_\emptyset \to X_{12}$ is required
 to be a
 weak equivalence. 
  
 Similarly, the above square is {\it homotopy cartesian} if the map from 
 $X_\emptyset$ to the homotopy pullback of $X_1 \to X_{12} \leftarrow X_2$ is
 a weak equivalence. When $X_2$ is contractible, we refer to $X_\emptyset \to X_1 \to X_{12}$ as a {\it homotopy fiber sequence}. The latter is equivalent to 
 describing a null homotopy of the composition $X_\emptyset \to X_1 \to X_{12}$ 
 such that the map from $X_\emptyset$ to the homotopy fiber of the map $X_1 \to X_{12}$
 is a weak equivalence.
 
 In each of these notions, when the null homotopy is understood, we typically omit it from the notation to avoid clutter.

\subsection{Factorization categories}
Fix a map of spaces $f: A\to B$. Define a category 
\[
T(f) = T(f\: A\to B)
\]
whose objects are spaces $X$ and a factorization $A\to X \to B$ by
continuous maps. A morphism $X\to X'$ is a map of spaces that is
compatible with the factorizations. When $f$ is understood,
we usually write this category as $T(A\to B)$.

Here are some important special cases:

\begin{ex} Let $A = B$ and use the identity map. Then
$T(B\to B)$ is the category of spaces which contain $B$ as a retract.
\end{ex}

\begin{ex} Let $A= \emptyset$ be the empty space. Then $T(\emptyset \to B)$ is the 
category of spaces over $B$.
\end{ex}

\begin{ex} Let $B = \ast$ be the one-point space. Then $T(A \to \ast)$
is the category of spaces under $A$.
\end{ex}

The forgetful functor $T(A \to B) \to T$ induces a  model structure on 
$T(A \to B)$ by declaring a morphism to be a cofibration, fibration  or weak equivalence if and only if it is one in $T$ \cite[2.8, prop.~6]{Quillen}. This
model structure enriched over $T$. 
The category of weak equivalences is denoted by
\[
wT(A \to B)\, .
\]

\begin{rem} We use the notation 
$|\cal C|$ for the {\it realization} (of the nerve) of a small category $\cal C$.
The functor $\cal C \mapsto |\cal C|$ enables one
to transfer homotopical properties of spaces over to small 
categories. For example, we
 declare a functor $f\: \cal C\to \cal D$ to be $r$-connected if and only if
it is so upon taking realization. Likewise, it makes
sense to ask whether a commutative square of small categories is homotopy cartesian.

In this paper the 
categories $\cal C$ that we will want to apply realization to
are full subcategories of $wT(A\to B)$---but they are not small.
This is not a major dilemma:
for a discussion of the options
on how to deal with the matter, see \cite[p.~766]{GK1}.
\end{rem}

\subsection{Spectra}  
The spectra appearing in this paper are formed from objects of $T_\ast$. 
For us, a spectrum will be a sequence of based spaces $E_j$ and (structure) maps $\Sigma E_j \to E_{j+1}$.
We say that $E$ is {\it cofibrant} if each of the spaces $E_j$ is cofibrant
and each structure map is a cofibration.  $E$ is {\it fibrant} if 
each adjoint $E_j\to \Omega E_{j+1}$ is a weak equivalence.

A map of spectra $f\: E\to E'$ is a collection of maps $f_j\: E_j \to E'_j$ that
are compatible with the structure maps. Any spectrum $E$ has a {\it fibrant replacement}, which is a spectrum $E^{\text{\rm f}}$ 
equipped with a natural map of spectra $E \to E^{\text{f}}$, where 
$E^{\text{f}}_j := \colim_k \Omega^k E_{j+k}$.  The map $f\: E \to E'$ is a {\it (stable) weak equivalence}
if the associated map $E^{\text{f}}\to (E')^{\text{f}}$ is such that 
for each $j$ the map of based spaces
 $E_j^{\text{f}}\to (E')^{\text{f}}_j $ is a weak equivalence. If $E$ is
a  spectrum, we write $\Omega^\infty E$ for the associated infinite loop space
given by the zeroth space of its fibrant replacement. If $X$ is a based space,
then we let $\Sigma^\infty X$ be its suspension spectrum whose $j$-th space
is $S^j \smsh X$. For it to have the correct homotopy type we should assume that
$X$ is cofibrant. The zeroth space of $\Sigma^\infty X$ is denoted $Q(X)$.

Given a based space $X$ and a spectrum $E$ we can form $X \smsh E$ which is the spectrum
whose $j$-th space is $X\smsh E_j$. This has the correct homotopy type if both
 $X$ and $E$ are cofibrant.
Similarly we can form the functions
$F_\ast(X,E)$ which is the spectrum whose $j$-th space is $F_\ast(X,E_j)$.
This has the correct homotopy type when $X$ is cofibrant and $E$ is fibrant (when
$E$ fails to be fibrant, we will implicitly replace it by its fibrant model).
The associated {\it stable function space} is $\Omega^\infty F_\ast(X,E)$. We 
will typically be sloppy and omit the $\Omega^\infty$ from the notation.
Thus, $F_\ast(X,E)$ can mean either the spectrum or its associated infinite loop space.
If $X$ is unbased then we set $F(X,E) = F_\ast(X_+,E)$ where $X_+ = X \amalg \ast$.
If $X$ and $Y$ are based spaces, then a {\it stable map} $X \to Y$
is an element of the stable function space $F_\ast(X,\Sigma^\infty Y)$, i.e.,
a point of the function space $F_\ast(X,Q(Y))$, where $Q(Y) = 
\Omega^\infty\Sigma^\infty Y$.
We let $\{X,Y\}_\ast$ denote the stable homotopy classes of maps from $X$ to $Y$; this
is the same as $\pi_0(F_\ast(X,Q(Y))$ when $X$ and $Y$ are cofibrant.

Smash products of spectra are barely used in this paper,
and are confined to the proof of Theorem \ref{bigthm:EIpartial}. 
It is for this reason that we are content to work in the above  category of spectra.
The reader is free to use a more modern approach.

\subsection{Spectra with group action} Fix a  discrete group $G$. 
We say that a spectrum $E$ has a $G$-action
if each
$E_j$ has the structure of a based $G$-space and each structure
map $\Sigma E_j \to E_{j+1}$ is equivariant, where $G$ acts
trivially on the suspension coordinate. A map $E\to E'$ of spectra
with $G$-action is just a map of underlying spectra which is $G$-equivariant.
A map of spectra is a {\it weak equivalence} if it is when considered as a map
of spectra without action.
We say that $E$ is fibrant
if its underlying spectrum (without action) is.  Call a
based $G$-space $X$ {\it $G$-cofibrant} if it is built up from the basepoint by 
attaching free $G$-cells along equivariant maps;
a free $G$-cell has the form $D^n {\times} G$.

If $X$ is an $G$-space and $E$ is a spectrum with $G$-action, then $G$ acts diagonally
on $X\smsh E$. We write $X\smsh_G E$ for the orbit spectrum.  This has the correct
homotopy type if $X$ and $E$ are both $G$-cofibrant.
The {\it homotopy orbits} of $G$ acting on $E$ is the spectrum
\[
E_{hG} = EG_+ \smsh_G E\, 
\]
where $EG$ the universal contractible $G$-space.
This has the correct homotopy type of 
the underlying spectrum of $E$ is cofibrant.

\subsection{Poincar\'e spaces}
The Poincar\'e spaces of this paper are orientable.
An orientable {\it Poincar\'e space} of dimension $d$ consists
of a homotopy finite space $P$ for which there exists 
a fundamental class $[P] \in H_d(P;\Bbb Z)$ such that
the cap product
\[
\cap [P] \: H^*(P;\cal M) \to H_{d-\ast}(P;\cal M)
\]
is an isomorphism in all degrees for any locally constant sheaf $\cal M$.

If $\pi\: \tilde P \to P$ is a choice
of universal cover, then the cap product is an isomorphism for all
$\cal M$  if and only if
it is an isomorphism for the locally constant sheaf $\Lambda$
whose stalk at $x\in P$ is given by the free abelian group
with basis $\pi^{-1}(x)$ (cf.\ \cite[lem.~1.1]{Wall_poincare}).

Poincar\'e spaces $P$ with boundary $\partial P$, also known as Poincar\'e  pairs,
are defined similarly, where now
$[P] \in H_d(P,\partial P;\Bbb Z)$,
the cap product
\[
H^\ast(P;\cal M) @>\cap [P] >> H_{d-\ast}(P,\partial P; \cal M)
\]
is an isomorphism, and the class
$[\partial P]
\in  H_{d-1}(\partial P;\Bbb Z)$, obtained by applying the boundary
homorphism to $[P]$, equips $\partial P$ the structure of a Poincar\'e space of 
dimension $d-1$
(this assumes in particular that $\partial P$ is homotopy finite). 
We will be relaxed about language and refer to 
a Poincar\'e space with or without boundary simply as a Poincar\'e space.

We will also sometimes omit the condition that
the map $\partial P \to P$ is an inclusion. The definition of a
Poincar\'e space still makes sense in this instance since we
can replace any map by its mapping cylinder inclusion.

\section{Poincar\'e embeddings \label{sec:embeddings}}

Let $P$ be a Poincar\'e space of dimension $n$. We will assume
here that $\partial P \to P$ is a cofibration.
An  {\it (interior) Poincar\'e embedding} of $P$ in $D^n$ consists of a space $C$
and a map $\partial P \amalg S^{n-1} \to C$ such that 
\begin{itemize}
\item $C$ is a Poincar\'e space with boundary $\partial P \amalg S^{n-1}$;
\item  the amalgamated union
\[
P \cup_{\partial P} C 
\]
is weakly contractible.
\end{itemize}

In what follows we set
\[
A := \partial P \amalg S^{n-1}\, .
\]

Then 
\[
C\in wT(A\to \ast)
\]
is an object. Let
\[
\cal E(P,D^n) \subset wT(A\to \ast)
\]
be the full subcategory whose objects
give Poincar\'e embeddings of $P$ in $D^n$. 
The space of Poincar\'e embeddings of $P$ in $D^n$ is then defined as
the realization
\[
E(P,D^n) = |\cal E(P,D^n) |\, .
\]
This is an open and closed subspace of $|wT(A\to \ast )|$.

\begin{rem} The version of the Poincar\'e embedding space appearing here
is slightly different from the one
in \cite[defn.~2.8]{GK1}. There it is defined to be the homotopy fiber
of the functor
\begin{equation}\label{eqn:interior}
\cal I(A) \to \cal I(S^{n-1})
\end{equation}
given by ``gluing in $P$,'' where for a Poincar\'e space $\partial$ without boundary,
the category $\cal I(\partial)$ has objects Poincar\'e spaces $X$ with $\partial$ as boundary,
and morphisms are weak homotopy equivalences $X \to X'$ which restrict to the identity on 
$\partial$. In the definition of  \cite[defn.~2.8]{GK1}, the homotopy fiber
of \eqref{eqn:interior} is taken at $D^n \in \cal I(S^{n-1})$.
Our definition here amounts to taking a certain open and closed subspace
of $\cal I(A)$ rather than a homotopy fiber. This definition is equivalent to
the one in \cite{GK1} because the component of $D^n \in \cal I(S^{n-1})$
is contractible.
\end{rem}

The {\it decompression functor}
\[
\cal E(P,D^n) \to \cal E(P{\times} D^1,D^{n+1})
\]
is defined by mapping $C\in \cal E(P,D^n)$ to its undreduced suspension $SC$.
On realizations it defines the {\it decompression map} 
\begin{equation} \label{eqn:decompression}
E(P,D^n) \to E(P{\times} D^1,D^{n+1})\, .
\end{equation}

\subsection{Unlinked embeddings}
If $P$ is sectioned, then we set
\[
A' = P^\xi \vee S^{n-1}
\]
There is then a cofiber sequence $P_+ \to A \to A'$.
The map $A \to A'$ induces a (forgetful) functor
\[
wT(A' \to \ast) \to wT(A \to \ast)\, .
\]

\begin{defn}
The space of {\it unlinked embeddings} 
$\frak L E(P,D^n)$ is the realization of the
 full subcategory
\[
\frak L\cal E(P,D^n) \subset wT(A'\to \ast)
\]
consisting of objects $C$ which become Poincar\'e embeddings
when considered in  $wT(A \to \ast)$. 
\end{defn}

Unraveling the definition, we see that an unlinked 
embedding consists of a space $C$ and a map $P^\xi \vee S^{n-1} \to C$
such that the  composition
\[
\partial P \amalg S^{n-1} \to P^\xi \vee S^{n-1} \to C
\]
defines a Poincar\'e embedding of $P$ in $D^n$.

By definition, there is a homotopy cartesian square
\begin{equation} \label{eqn:e'e-pullback}
\xymatrix{
\frak L \cal E(P,D^n) \ar[r]\ar[d] & wT(A'\to \ast) \ar[d] \\
\cal E(P,D^n) \ar[r] & wT(A\to \ast) \, .
}
\end{equation}

\subsection{Poincar\'e Immersions}
Recall $P_j := P \times D^j$. The
{\it immersion space} $I(P,D^n)$
is defined as the homotopy colimit of the sequence of decompression maps
\[
E(P,D^n) \to E(P_1,D^{n+1})  \to E(P_2,D^{n+2}) \cdots\, .
\]
 \begin{lem} \label{lem:immersion-stable-normal-invt} There is a homotopy equivalence
 \[
 I(P,D^n) \simeq \Omega^{n}_{\perp} Q(P/\partial P) \, .
 \]
 \end{lem}

\begin{proof}  The Browder construction gives a factorization of the
filtration defining $I(P,D^n)$ as
{\small\[
 E(P,D^n) \to \Omega^n_{\perp} P/\partial P \to E(P_1,D^{n+1}) \to 
\Omega^{n+1}_\perp P_1/\partial P_1 \to E(P_2,D^{n+2}) \to \dots
\]}
The homotopy colimit of the odd terms appearing in the sequence yields $I(P,D^n)$ by 
definition, whereas the homotopy colimit of the even terms  gives the space of stable
normal invariants
\[
\Omega^n_{\perp} Q(P/\partial P)\, ,
\]
since $\Omega^{n+j} (P_j/\partial P_j) = \Omega^{n+j}\Sigma^{j}(P/\partial P)$.
\end{proof}

\begin{lem}[Smale-Hirsch for Poincar\'e Spaces] \label{lem:smale-hirsch}
If $I(P,D^n)$ is non-empty, then 
there is a preferred weak homotopy
equivalence
\[
I(P,D^n) \simeq F(P,G)\, ,
\]
where the right side denotes the function space of unbased maps from 
$P$ to the topological monoid of stable self equivalences of the sphere.
\end{lem} 

\begin{rem} The corresponding statement
in the smooth case is that the smooth immersions of $P$ to $D^n$ is 
weak equivalent to the function space $F(P,O_n)$, where $O_n$ is the
group of orthogonal $n\times n$ matrices. Note that the smooth version depends
on $n$.

To obtain a smooth statement  which does not depend on $n$, one
should replace the smooth immersion space
by its block analogue.
In this instance one obtains a weak equivalence to the
 function space $F(P,O)$.
\end{rem}

\begin{proof}[Proof of Lemma \ref{lem:smale-hirsch}] 
If $I(P,D^n)$ is non-empty then the Spivak fibration
for $P$ is trivializable, implying  that
$P/\partial P$ is $n$-dual to $P_+$. By S-duality, we have a weak equivalence
\[
F_\ast(P_+,Q(S^0)) \simeq \Omega^n Q(P/\partial P)\, .
\]
Restricting to stable normal invariants on the right corresponds to replacing 
$Q(S^0)$ on the left by its units, namely $G$. The result now follows by Lemma
\ref{lem:immersion-stable-normal-invt}. \end{proof}   

\section{Proof of Theorems 
and \ref{bigthm:browder-fibration} and \ref{bigthm:refined_browder-construction}
\label{sec:proofs}}

Consider the following situation: fix a map
of based spaces 
\[
f\: A\to X
\] and let  $Y$ denote its reduced mapping cone. 
Consider the forgetful functor
\[
wT(Y\to \ast) \to wT(X\to \ast)\, .
\]
Let $Z \in wT(X\to \ast)$ be an object; in particular,
$Z$ has the structure of a based space.
The map $X \to Z$ factors through $Y$
 precisely when the composition
\[
A \to X \to Z
\]
is null homotopic.  In what follows we fix based null homotopy
$CA \to Z$.

\begin{lem} \label{lem:helper_sequence}  With these assumptions, there
is a homotopy fiber sequence
\[
F_\ast(\Sigma A,Z) @>>> |wT(Y\to \ast)| \to |wT(X\to \ast)|\, ,
\]
where the displayed fiber is taken at the basepoint
 $Z\in |wT(X\to \ast)|$.
\end{lem}

\begin{proof} By \cite[prop.~2.14]{GK1}, the homotopy
fiber of $|wT(Y\to \ast)| \to |wT(X\to \ast)|$ taken at $Z$
identified up to preferred weak equivalence with the function space of liftings/extensions 
\[
\xymatrix{
X \ar[r] \ar[d] & Z \ar[d] \\
Y \ar@{.>}[ur]\ar[r]  & \ast .
}
\]
We employ the notation
\[
F_X(Y,Z)
\]
for this space. The given extension equips $F_X(Y,Z)$ with a basepoint and the restriction map $F_X(Y,Z) \to F_{A}(CA,Z)$ is a homeomorphism. Let
$\rho\: CA \to Z$ be the given null homotopy.
The cofiber sequence $A \to A \to CA$ has a coaction map $\delta\: 
CA \to CA \vee \Sigma A$.
Given a map $g\: \Sigma A \to Z$, we form
\[
\rho\star g\: CA @>\delta >> CA \vee \Sigma A @> \rho + g >> Z\ .
\]
Then $g\mapsto \rho\star g$ defines a weak equivalence
 $F_\ast(\Sigma A,Z) \simeq F_{A}(CA,Z)$.
\end{proof}

\begin{proof}[Proof of Theorem \ref{bigthm:refined_browder-construction}]
Recall that $A = \partial P \amalg S^{n-1}$ and $A' = P^\xi \vee S^{n-1}$.
Consider the full subcategory 
\[
wT(A'\to \ast;{\sim}P^\xi)\subset wT(A' \to \ast)
\] 
with objects $C$ such that the composite
\[
P^\xi  \subset A' \to C
\]
is a weak homotopy equivalence. We  claim there is a homotopy equivalence
\[
|wT(A'\to \ast;{\sim} P^{\xi} )| \simeq \Omega^{n-1} P^\xi\, .
\]
To see this, note that $wT(A'\to \ast;{\sim} P^\xi )$ is the right fiber taken
at $P^\xi\in wT(\ast\to \ast)$ of the forgetful functor
\[
wT(S^{n-1}\to \ast) \to wT(\ast\to \ast)
\]
and by \cite[prop.~2.19]{GK1} we may identify this right fiber with
$\Omega^{n-1} P^\xi$. This gives the claim. 

We restrict our attention to the full subcategory
\[
wT_\perp(A'\to \ast;{\sim} P^\xi )\subset wT(A'\to \ast;{\sim} P^\xi )
\] of those objects
$C$ such that the weak map $S^{n-1} \to C @< \sim << P^\xi$ corresponds to an
unstable normal invariant. This additional
constraint yields a homotopy equivalence
\[
|wT_\perp(A'\to \ast;{\sim} P^\xi )| \simeq \Omega^{n-1}_\perp P^\xi\, .
\]
The refined Browder contruction defines a functor
\[
F\: wT_\perp(A'\to \ast;{\sim} P^\xi ) \to \frak L\cal E(P,D^n)
\]
On the other hand, the identity defines a functor
\[
G\: \frak L\cal E(P,D^n)\to wT_\perp(A'\to \ast;{\sim} P^\xi )\, .
\]
It is tautological 
that these functors are inverses to each other.
\end{proof}

\begin{proof}[Proof of Theorem \ref{bigthm:browder-fibration}]
By Theorem \ref{bigthm:refined_browder-construction}
it suffices to consider the map
\[
\frak L E(P,D^n) \to E(P,D^n)\, .
\]
We will make use of the cofiber sequence
\[
P_+ \to A\to A'\, .
\]
If $C \in \cal E(P,D^n)$ is an object, then clearly the obstruction to lifting
it to an object of $\frak L\cal E(P,D^n)$ up to weak equivalence
 is that the composite
\[
P_+ \to A \to C
\]
is null homotopic.  This proves the first part. Now suppose a null homotopy
$P_+ \to C$ has been chosen. Using Lemma \ref{lem:helper_sequence}, we have
a homotopy fiber sequence
\[
F_\ast(\Sigma(P_+),C) \to |wT(A'\to\ast)| \to |wT(A\to\ast)|\, .
\]
One completes the proof using the homotopy cartesian square
\eqref{eqn:e'e-pullback}.
\end{proof}

\section{Proof of Theorem \ref{bigthm:EIpartial}\label{sec:proof-EIpartial}}

\subsection{Principal fibrations} We recall a
basic result about principal fibrations from 
\cite[lem.~6.1]{Klein_susp_spectra}.
Suppose 
$p\:E \to Z$ is a fibration.
We say that $p$ is {\it principal} if there
exists a commutative homotopy cartesian square
of  spaces
\[
\xymatrix{
E \ar[r] \ar[d]_{p} & C\ar[d]  \\
Z\ar[r]  &B
}
\]
such that $C$ is contractible.  Note that the property of being
a principal fibration is
preserved under base changes. Choose a basepoint for $C$. This gives
a basepoint for $B$.
\medskip

Suppose that
$Z$ is connected. If $p\: E\to Z$ is principal, there
is an ``action'' $\Omega B \times E \to E$. 
If there exists a section $Z \to E$, one can combine it
with this action to produce a map
of fibrations $\Omega B \times Z \to E$
covering the identity map of $Z$. This implies that $p$ is
weak fiber homotopically trivial. Let $\secs(p)$
denote the space of sections of $p$. Then we have shown

\begin{lem}\label{lem:section} Assume $p\: E \to Z$ is principal. 
Assume that $\secs(p)$
is non-empty and comes equipped with basepoint.
Then there is a preferred weak equivalence 
$
\secs(p)\,\, \simeq \,\, F(Z,\Omega B)
$.
\end{lem}

Let $p \: E \to Z$ be a principal fibration and suppose
that $A \to Y$ is a cofibration
as in Lemma \ref{lem:section}. Then given a lifting problem
\[
\xymatrix{
A \ar[d]\ar[r] & E \ar[d]^p \\
Y \ar[r]_f\ar@{.>}[ur]& Z
}
\]
we let $\lift(f|p)$ be the solution space:  the space of maps $Y\to E$ of $f$ making the diagram commute.

\begin{cor} \label{cor:section} If $\lift(f|p)$ is non-empty then a choice
of lift determines a weak equivalence
$
\lift(f|p)\simeq F_\ast(Y/A,\Omega B)
$.
\end{cor}

\begin{proof} Observe that $f^*E \to Y$ is principal.
Furthermore $\lift(f|p) \cong \secs(f^*E \to Y)$. Hence if $\lift(f|p)$ is nonempty
we can identify $f^* E$ with the trivial fibration $\Omega B \times Y \to Y$
once a basepoint lift has been chosen.
With respect to the identification, the given map $A \to f^*E$ corresponds
to the inclusion $\ast \times A \to \Omega B \times Y$.

Hence,  $\lift(f|p)$ is then identified up to weak equivalence
with  the space of sections of the 
trivial fibration $\Omega B \times Y \to Y$ which are fixed on $A$.
But this is just  $F_\ast(Y/A,\Omega B)$.
\end{proof}

\subsection{The Goodwillie tower of the identity}
Let $\Bbb I\: \Top_\ast\to \Top_\ast$ be the identity functor.
We recall some of the basic properties of its Goodwillie tower.
(cf.\ \cite{Good_calc1},
\cite{Good_calc2}, \cite{Good_calc3}, 
\cite{Johnson-deriv}).

\begin{thm} 
\label{thm:good-identity}
There is a tower of fibrations
of homotopy functors on based spaces
$$
\cdots \to P_2\Bbb I(X) \to P_1\Bbb I(X) 
$$
and compatible natural transformations
$X \to P_j\Bbb I(X)$ such that
\begin{itemize} 
\item if $X$ is $1$-connected, then the natural map
$$
X \to \lim_j P_j\Bbb I(X)
$$
is a weak equivalence.
\item There is a natural weak equivalence
$P_1\Bbb I(X) \simeq Q(X)$;
\item For $j\ge 2$, the fibration $P_j\Bbb I(X) \to P_{j{-}1}\Bbb I(X)$
is principal (cf. \cite[lem~2.2]{Good_calc3};
\item the $j$-th layer $L_j\Bbb I(X) := \text{\rm fib}(P_j\Bbb I(X) \to P_{j{-}1}\Bbb I(X))$ is naturally
weak equivalent to the functor 
$$
X \mapsto \Omega^\infty (\cal W_j \smsh_{h \Sigma_j} X^{[j]}) \,  ;
$$
where the spectrum with $\Sigma_j$-action $\cal W_j$ is  as in \cite{Johnson-deriv}.
\end{itemize}
\end{thm}

\subsection{The spaces $E_j(P,D^n)$} Recall 
that $P$ is sectioned by $\xi$.
We fix the natural identification
$P_1\Bbb I(X) \simeq Q(X)$. Then we have a map
\begin{equation} \label{eqn: normal_ivariant_to_P1}
\Omega^{n-1}_{\perp}Q(P^\xi) \to \Omega^{n-1}P_1\Bbb I(P^\xi)\, .
\end{equation}
Note that the source of this map is identified with a 
collection of components of the target.

\begin{defn} The space $\frak L E_j(P,D^n)$ is defined to be the pullback of the diagram
\[
\Omega^{n-1}_{\perp}Q(P^\xi) 
@>>> \Omega^{n-1} P_1\Bbb I(P^\xi) @<<< \Omega^{n-1} P_j\Bbb I(P^\xi)\, .
\]
\end{defn}

\begin{proof}[Proof of Theorem \ref{bigthm:EIpartial}]
 It is a consequence of the definition that
there is a tower of fibrations
\begin{equation}\label{eqn:towerLE}
\cdots @>>> \frak L E_2(P,D^n) @>>> \frak L E_1(P,D^n) \, .
\end{equation}
By definition $\frak L E_1(P,D^n) \cong \Omega^{n-1}_{\perp}Q(P^\xi)$
and by Lemma \ref{lem:immersion-stable-normal-invt}, 
$\Omega^{n-1}_{\perp}Q(P^\xi) \simeq I(P,D^n)$.
Moreover, the square
\begin{equation} \label{eqn:key-diagram}
\xymatrix{
\lim_j \frak L E_j(P,D^n) \ar[r] \ar[d] & 
\lim_j \Omega^{n-1} P_j\Bbb I(P^\xi) \ar[d]\\
\frak L E_1(P,D^n) \ar[r] & P_1\Bbb I(P^\xi)
}
\end{equation}
is homotopy cartesian. The lower right corner of
this diagram is identified with $Q(P^\xi)$.
Since that map $\partial P \to P$ is at least $2$-connected,
it follows that the section $K \to \partial P$ is at least $1$-connected. Hence,
$P^\xi$ is $1$-connected and the upper right corner of diagram
 is identified with $ \Omega^{n-1} P^\xi$.  
 Substituting these identifications,
 we obtain a homotopy cartesian square
  \begin{equation} \label{eqn:key-diagram2}
\xymatrix{
\lim_j \frak L E_j(P,D^n) \ar[r] \ar[d] & \Omega^{n-1}P^\xi \ar[d]\\
I(P,D^n) \ar[r] & \Omega^{n-1} Q(P^\xi)\, .
}
\end{equation}
Clearly, if we replace the upper left corner by the space of unstable normal
invariants $\Omega^{n-1}_\perp P^\xi$ the square remains homotopy cartesian,
since a point of $\Omega^{n-1}P^\xi$ yields an unstable normal invariant
if and only if the associated point of $\Omega^{n-1} Q(P^\xi)$ yields
a stable normal invariant.
It follows that the map  
\[
\Omega^{n-1}_\perp P^\xi \to \lim_j \frak L E_j(P,D^n)
\]
is  a weak equivalence. Therefore, the composite map
\[
\frak LE(P,D^n) @> \sim >> \Omega^{n-1}_\perp P^\xi  @> \sim >> \lim_j \frak L E_j(P,D^n)
\]
is also a weak equivalence.

We next identify the layers of the tower \eqref{eqn:towerLE} whenever they are
non-empty. 
The fiber
of the map $\Omega^{n-1} P_j\Bbb I(P^\xi) \to \Omega^{n-1}  P_{j-1}\Bbb I(P^\xi)$
at any basepoint is just the lifting space
\begin{equation}\label{eqn:lifting_space}
\xymatrix{
\ast \ar[r] \ar[d] & P_j\Bbb I(P^\xi) \ar[d] \\
S^{n-1} \ar[r] \ar@{.>}[ur] & P_{j-1}\Bbb I(P^\xi)\, .
}
\end{equation}
If this lifting space is non-empty, then Corollary \ref{cor:section} says that
after making a choice of lift, the
space of all such lifts 
is identified with the stable function space
\[
F_\ast(S^{n-1},\cal W_j\smsh_{h\Sigma_j} (P^\xi)^{[j]})\, .
\]
We need to rewrite this stable function space 
up to homotopy in the requisite form. 
First identify it as the zeroth space
of the homotopy orbit spectrum
\begin{equation}\label{eqn:first}
(S^{1-n} \smsh \cal W_j\smsh (P^\xi)^{[j]}))_{h\Sigma j}\, .
\end{equation}
The plan is to rewrite the latter in terms of $P_+$ using Spanier-Whitehead duality.
Assuming that $I(P,D^n)$ is non-empty guarantees that
$P^\xi$ is $(n-1)$-dual to $P_+$.  We can write this as
\[
\Sigma^\infty P^\xi \simeq \Sigma^{n-1} D(P_+) \, ,
\]
where $D(P_+)$ is the $0$-dual of $P_+$, i.e.,
 $F_\ast(P_+,S)$, where $S$ denotes the sphere spectrum.
If we smash this identification with itself $j$-times, we obtain an equivariant weak equivalence of spectra with $\Sigma_j$-action
\[
\Sigma^\infty (P^\xi)^{[j]} \simeq  F_\ast((P^{\times j})_+,\Sigma^\infty S^{(n-1)(V_j+1)}) \, .
\]
Substituting this into \eqref{eqn:first}, and doing some minor rewriting,
we obtain the spectrum
\[
F_\ast((P^{\times j})_+,\cal W_j\smsh S^{(n-1)V_j} )_{h\Sigma_j}\, .
\]
The zeroth space of this spectrum is thus identified with
the homotopy fibers of $\frak LE_j(P,D^n) \to \frak LE_{j-1}(P,D^n)$
whenever these are non-empty.

Lastly, we need to exhibit the obstruction $\ell_{j-1}$.  
According to \cite[lem.~2.2]{Good_calc3} there is a
$j$-homogeneous functor $X\mapsto R_j\Bbb I(X)$ and a homotopy cartesian square
\[
\xymatrix{
P_j \Bbb I(X) \ar[r] \ar[d] & C\ar[d]  \\
P_{j-1} \Bbb I(X) \ar[r] &
R_j\Bbb I(X)
}
\]
where $C$ is contractible. By the classification of $j$-homogeneous functors \cite{Good_calc3}, \cite[p.~5]{Good_calc1},
we infer
\[
R_j\Bbb I(X) \simeq \Omega^\infty \cal V_j \smsh_{h\Sigma_j}X^{[j]}
\]
for a some spectrum with $\Sigma_j$-action $\cal V_j$. Furthermore, the  map
\[
L_j\Bbb I(X) \to \Omega R_j\Bbb I(X)
\]
is a weak equivalence of $j$-homogenous functors. It follows that there
is a weak equivalence of spectra with $\Sigma_j$-action
\[
\cal V_j \simeq \Sigma \cal W_j\, .
\]
Consequently, setting $X = P^\xi$,
 we have a homotopy cartesian square
 \[
\xymatrix{
P_j \Bbb I(P^\xi) \ar[r] \ar[d] & C\ar[d]  \\
P_{j-1} \Bbb I(P^\xi) \ar[r] &
\Omega^\infty \Sigma\cal W_j \smsh_{h\Sigma_j}(P^\xi)^{[j]}
}
\]
in which $C$ is contractible. Hence, the obstruction up to homotopy to a lifting
a based map $x\: S^{n-1} \to P_{j-1} \Bbb I(P^\xi)$ to 
a based map $S^{n-1} \to P_{j} \Bbb I(P^\xi)$ is given by 
the homotopy class of the  composition
\begin{equation} \label{eqn:ell}
 S^{n-1} @>>> P_{j-1} \Bbb I(P^\xi) \to 
\Omega^\infty \Sigma\cal W_j \smsh_{h\Sigma_j}(P^\xi)^{[j]}\, .
\end{equation}
In particular, if  
$x\in E_{j-1}(P,D^n) \subset  \Omega^{n-1} P_{j-1} \Bbb I(P^\xi)$
is a point, then
we define $\ell_{j-1}(x)$ to be the homotopy class of 
\eqref{eqn:ell}. Then $\ell_{j-1}(x)$ {\it a priori} lies in the abelian group
\[
\{ S^{n-1},\Sigma \cal W_j \smsh_{h\Sigma_j} (P^\xi)^{[j]}\}_\ast
\]
Again by duality, we can rewrite the latter up to canoncial isomorphism as
\[
\pi_0(F_\ast(P^{\times j}_+,\cal W_j\smsh S^{(n-1)(V_j-1)+1})_{h\Sigma_j}) \, . \qedhere
\] 
\end{proof}

\section{Appendix: relationship with manifold calculus \label{sec:relationship}}
It is legitimate to ask what
the tower of Theorem \ref{bigthm:EIpartial} has to do with the
Goodwillie-Weiss manifold calculus \cite{Weiss}. 
Here is one possible scenario: 
suppose that $Q$ is a compact smooth $(n-1)$-manifold---which we assume
admits a handle decomposition
with handles of index at most $k\le n-4$.
We consider the forgetful/decompression map
\begin{equation} \label{eqn:compare}
E^{\sm}(Q,D^{n-1}) \to \frak LE(Q{\times} D^1,D^{n})
\end{equation}
from the space of smooth embeddings of $Q$ in $D^n$
to the space of unlinked Poincar\'e embeddings of $Q{\times} D^1$ in $D^{n}$.

\begin{conjecture} \label{conj:compare} The map
\eqref{eqn:compare} induces a map of towers
from the Goodwillie-Weiss tower for $E^{\sm}(Q,D^{n-1})$
to the tower of Theorem \ref{bigthm:EIpartial} for $\frak LE(Q{\times} D^1,D^{n})$.
\end{conjecture}

We will give some evidence for this conjecture on the level of layers.
In what follows we shall assume that the reader is familiar with 
\cite{Weiss}.
Here is some notation: suppose $T$ is a finite set. We write 
\[
\cal P(Q,T)
\]
for the configuration space of the injective functions from
$T$ to the interior of $Q$.
This has a free action of $\Sigma_T$, the symmetric group of automorphisms
of $T$. In the case when $T =\underline j = \{1,\dots,j\}$, we write 
$\cal P(Q,j) := \cal P(Q,\underline j)$.
Let \[
\tbinom{Q}{T} := \cal P(Q,T)_{\Sigma_T}
\] 
the orbit space of the action of $\Sigma_T$ on $\cal P(Q,T)$;
this is the configuration space of (unordered) subsets of the interior
of $Q$ of cardinality $|T|$. Similarly, we write 
$\tbinom{Q}{j} := \tbinom{Q}{\underline j}$.

The quotient map $\pi\: \cal P(Q,j) \to \binom{Q}{j}$ is a 
principal covering space with respect to the group $\Sigma_j$;
we let $i\: \binom{Q}{j}\to B\Sigma_j$ denote its classifying map.
Then $i\circ \pi$ is the constant map to the basepoint. Let 
$c\: Q^{\times j} \to B\Sigma_j$ be the constant map. Then
$c$ restricted to $\cal P(Q,j)$ is $i \circ \pi$. There is an inclusion
$\cal P(Q,j) \subset Q^{\times j}$ whose complement is the diagonal
$\Delta \subset Q^{\times j}$.

The evidence we give for Conjecture \ref{conj:compare} 
is a diagram
{\small \[
\xymatrix{
H_{\text{cs}}^\bullet(\tbinom{Q}{j}; \cal E_j)\ar[r]^{(a)} &
H_{\text{cs}}^\bullet(\tbinom{Q}{j}; \cal F_j)\ar[d]^{(b)} \\
& H_{\text{cs}}^\bullet(\tbinom{Q}{j}; i^*\cal G_j) \ar[d]_{\simeq}^{(c)} \\
& H_{\text{cs}}^\bullet(\cal P(Q,j);(i\circ \pi)^*\cal G_j)_{h\Sigma_j} \\
& H^{\bullet}(Q^{\times j}, 
\Delta ; c^\ast\cal G_j)_{h\Sigma_j}  \ar[r]^(.55){(e)} \ar[u]^{\simeq}_{(d)}&
H^{\bullet}(Q^{\times j}; c^*\cal G_j)_{h\Sigma_j}\, .
}
\]}
We first summarize what the maps of the diagram are about and thereafter
we give
some of the details. For $j \ge 2$,
the source of (a) is the $j$-th layer for the manifold calculus tower
of the smooth embedding space $E^{\sm}(Q,D^{n-1})$. The target
of (a) is the $j$-th layer for the manifold calculus tower of, in the terminology
of \cite[defn.~2.2]{Weiss},
the good cofunctor $\cal O \mapsto \Omega^{n-1}\Sigma^{n-1}\cal O^+$, where
$\cal O \subset Q$ varies over the open subsets of the interior of $Q$
and $\cal O^+$ denotes  one-point compactification. The map (a) is
induced by the Pontryagin-Thom construction.
The map
(b) is a kind of stabilization map. The equivalence (c) is a version of the Adams
isomorphism (which is valid since the source and target in this case
are infinite loop spaces)
and the equivalence (d) is excision. The map (e) is a relaxation of constraints.
The target of (e) coincides with the $j$-th layer of the tower
of  Theorem \ref{bigthm:EIpartial} for $\frak LE(Q{\times} D^1,D^n)$.

We now proceed to give more detail.
In the above, $\cal E_j$ and $\cal F_j$ are fibrations over $\tbinom{Q}{j}$,
and $\cal G_j$ is a fibration over $B\Sigma_j$. The notation
$H^\bullet(B;\cal U)$ refers to the space of sections 
of a fibration $\cal U \to B$
(which we feel compelled to indicate as ``unstable'' cohomology), 
and similarly, $H^\bullet_{\text{cs}}(B;\cal U)$
refers to the space of sections with compact support relative
to a given fixed section; these are the sections which agree with the
given one outside a compact subset of $B$.

The fibration $\cal E_j \to \binom{Q}{j}$ may be described as follows: if 
$T\in \binom{Q}{j}$, then we form the $j$-cube of spaces
\begin{equation}\label{eqn:E-fibration}
U \mapsto \cal P(D^{n-1},U), \qquad U \subset T\, .
\end{equation}
The total homotopy fiber of \eqref{eqn:E-fibration}
is the fiber at $T$ of the fibration $\cal E_j \to \binom{Q}{j}$ 
(we will leave it to the reader to provide the topology
on $\cal E_j$ as well as on $\cal F_j$).
By \cite[sum.~4.2]{Weiss}
we know that $H^\bullet_{\text{cs}}(\tbinom{Q}{j}; \cal E_j)$ is the 
$j$-th layer of the Goodwillie-Weiss tower for the space of smooth embeddings 
$E^{\sm}(Q,D^{n-1})$ when $j \ge 2$.

The fibration $\cal F_j \to \binom{Q}{j}$ has fiber at $T$ given by
 the total homotopy fiber
of the $j$-cube
\begin{equation}\label{eqn:F-fibration}
U \mapsto \Omega^{n-1}\Sigma^{n-1}(U_+), \qquad U \subset T\, .
\end{equation}
The map $\cal E_j \to \cal F_j$ is induced by the Pontryagin-Thom construction
$\cal P(D^{n-1},U) \to \Omega^{n-1}\Sigma^{n-1}(U_+)$ with respect to the trivial framing.
This induces the map (a). The map (a) is $(2n-3k-5)$-connected \cite{Klein_compress}.

The fibration $\cal G_j \to B\Sigma_j$
arises as follows: take the unreduced Borel construction of $\Sigma_j$ acting
on $\cal W_j \smsh S^{(n-1)V_j}$. This gives a fibered spectrum
\begin{equation} \label{eqn:fibered_spectrum}
E\Sigma_j \times_{\Sigma_j}(\cal W_j \smsh S^{(n-1)V_j}) \to B\Sigma_j\, .
\end{equation}
Then $\cal G_j$ is the fiberwise zeroth space of \eqref{eqn:fibered_spectrum}, i.e,
$\cal G_j$ is the unreduced Borel construction of $\Sigma_j$ acting on
$\Omega^\infty(\cal W_j \smsh S^{(n-1)V_j})$. 
The fibration $i^*\cal G_j \to \binom{Q}{j}$ is obtained by taking the base change
of $\cal G_j$ along $i$.

The target of (c) is the homotopy orbits of $\Sigma_j$ acting on 
the section space with compact supports  of the fibration 
$(i\circ \pi)^\ast\cal G_j \to \cal P(Q,j)$. Here we are using the observation
that this last map is $\Sigma_j$-equivariant ($\Sigma_j$-acts
on $(i\circ \pi)^\ast\cal G_j$ because it is a trivial 
fibration over $\cal P(Q,j)$ whose fiber $\Omega^\infty(\cal W_j \smsh S^{(n-1)V_j})$
comes equipped with
a $\Sigma_j$-action).
As already mentioned, the map (c) is a homotopy equivalence by 
the Adams isomorphism \cite[\S2]{May} and the map (d) is an equivalence by  excision.
The map (e) is the map which forgets that a section is fixed
along the diagonal; it is $((j{-}1)(n{-}2){-}k{-}1)$-connected.
 
The map (b) is induced by a map $\cal F_j\to i^\ast\cal G_j$
which on the level of fibers at $T$ arises from the evident stabilization (inclusion)
map 
\[
\Omega^{n-1}\Sigma^{n-1}(U_+) \to \colim_{n \to \infty} \Omega^{(n-1)+(n-1)V_j}
\Sigma^{(n-1)+(n-1)V_j}(U_+)
\]
which is a map of $j$-cubes.
As noted above, the total homotopy fiber of the source is identified with the
fiber of $\cal F_j$ at $T$ and an
 analysis which we omit shows that total homotopy fiber of the target 
is identified with the fiber of $i^\ast\cal G_j$ at $T$. In fact, the Hilton-Milnor theorem
shows that the map $\cal F_j \to i^\ast\cal G_j$ is $((j+1)(n-2)+2 - n)$-connected.
By subtracting the handle dimension of $Q^{\times j}$ (i.e., $jk$) 
we infer that (b) is  $j(n{-}k{-}2)$-connected.

Finally, observe that the connectivity of each of
the maps (b) and (e) is a linear function
of $j$ with positive slope. Thus the map (a) is the only
map of the diagram
which does not tend to weak equivalence as $j$ gets large.

\end{document}